\documentclass[11pt,a4paper,reqno]{amsart}
\usepackage{amsmath,amsthm,amsfonts}
\usepackage{paralist}
\usepackage{graphics}
\usepackage{epsfig}
\usepackage{cite}
\usepackage[a4paper]{geometry}

\usepackage{comment}
\usepackage{xcolor}
\usepackage{subfigure}
\usepackage{dsfont}
\usepackage{bbm}
\usepackage{bm}
\usepackage{soul}
\usepackage[shortlabels]{enumitem}

\usepackage{amssymb,amsxtra,setspace,xspace,lmodern,psfrag,color,latexsym}
\usepackage{mathtools,pifont,mathrsfs,caption,microtype,accents}

\usepackage[colorlinks=true]{hyperref}
\hypersetup{urlcolor=blue,citecolor=red,linkcolor=black}

\allowdisplaybreaks[4]

\newenvironment{listi}
{\begin{list}
    {(\roman{broj})}
    { \usecounter{broj}}
    \setlength{\labelwidth}{30pt}
    \addtolength{\itemsep}{2pt}}
  {   \end{list} }
\newcounter{broj}

\definecolor{darkgreen}{rgb}{0.1,0.6,0.1}
\definecolor{darkred}{rgb}{0.6,0,0}
\definecolor{lightgray}{rgb}{0.5,0.5,0.5}

\DeclareMathAlphabet{\mathup}{OT1}{\familydefault}{m}{n}
\newcommand{\di}[1]{\mathop{}\!\mathup{d} #1}
\newcommand{\dx}[1]{\mathop{}\!\mathup{d} #1}
\newcommand{\pderiv}[3][]{\frac{\mathop{}\!\mathup{d}^{#1} #2}{\mathop{}\!\mathup{d} #3^{#1}}}
\renewcommand{\:}{\colon}

\newcommand{\N}{\mathbb{N}}

\newcommand{\R}{\mathbb{R}}
\newcommand{\Rd}{{\mathbb{R}^{d}}}

\newcommand{\Rdd}{{\mathbb{R}^{2d}}}

\newcommand{\cM}{\mathcal{M}}

\newcommand{\cP}{\mathcal{P}}
\newcommand{\cV}{\mathcal{V}}

\newcommand{\bj}{\bm{j}}
\newcommand{\dgrad}{\overline\nabla}
\newcommand{\eps}{\varepsilon}

\newcommand{\cMtv}{\mathcal{\cM}_{\mathrm{TV}}}

\newcommand{\ACt}{\mathcal{AC}_T}
\newcommand{\dT}{\boldsymbol{d}_{\ACt}}
\newcommand{\cB}{\mathcal{B}}

\newtheorem{theorem}{Theorem}[section]
\newtheorem{corollary}[theorem]{Corollary}
\newtheorem{lemma}[theorem]{Lemma}
\newtheorem{proposition}[theorem]{Proposition}

\newtheorem{definition}[theorem]{Definition}

\theoremstyle{remark}
\newtheorem{rem}[theorem]{Remark}
\newtheorem{exm}[theorem]{Example}

\newtheoremstyle{namedthmstyle} %Name
{} %Space above (default, \topsep)
{} %Space below (default, \topsep)
{\itshape} %Body font
{} %Indent amount 1
{\bfseries} %Theorem head font
{} %Punctuation after theorem head
{ } %Space after theorem head 2
{}
\theoremstyle{namedthmstyle}
\newcommand{\thistheoremname}{}
\newtheorem*{genericthm}{\thistheoremname}
\newenvironment{namedthm}[1]
{\renewcommand{\thistheoremname}{#1}%
  \begin{genericthm}}
  {\end{genericthm}}
\makeatletter
\def\namedlabel#1#2{\begingroup
  \def\@currentlabel{#2}%
  \label{#1}\endgroup
}
\makeatother

\newcommand{\Rddiag}{\R^{2d}_{\!\scriptscriptstyle\diagup}}

\makeatletter
\@namedef{subjclassname@2020}{%
  \textup{2020} Mathematics Subject Classification}
\makeatother

\numberwithin{equation}{section}

\DeclareMathOperator{\I}{I}
\DeclareMathOperator{\II}{II}

\DeclareMathOperator{\AC}{AC}
\DeclareMathOperator{\TV}{{TV}}

\DeclareMathOperator{\CE}{CE}

\DeclareMathOperator{\supp}{supp}

\DeclarePairedDelimiter{\abs}{\lvert}{\rvert}
\DeclarePairedDelimiter{\norm}{\lVert}{\rVert}
\DeclarePairedDelimiter{\bra}{(}{)}

\DeclarePairedDelimiter{\set}{\{}{\}}

\title{On a Class of Nonlocal Continuity Equations on Graphs}

\begin{document}
\author{A. Esposito \and F. S. Patacchini \and A. Schlichting} 
\address{A. Esposito -- Mathematical Institute, University of Oxford, Woodstock Road, Oxford, OX2 6GG, United Kingdom.}
\address{F. S. Patacchini -- IFP Energies nouvelles, 1-4 avenue de Bois-Pr\'eau, 92852 Rueil-Malmaison, France}
\address{A. Schlichting -- Institute for Analysis and Numerics, University of Münster, Orléans-Ring 10, 48149 Münster, Germany}
\email{antonio.esposito@maths.ox.ac.uk}
\email{francesco.patacchini@ifpen.fr}
\email{a.schlichting@uni-muenster.de}

\begin{abstract}
Motivated by applications in data science, we study partial differential equations on graphs. By a classical fixed-point argument, we show existence and uniqueness of solutions to a class of nonlocal continuity equations on graphs. We consider general interpolation functions, which give rise to a variety of different dynamics, e.g., the nonlocal interaction dynamics coming from a solution-dependent velocity field. Our analysis reveals structural differences with the more standard Euclidean space, as some analogous properties rely on the interpolation chosen.
\end{abstract}

\keywords{evolution on graphs, flux interpolation, upwind interpolation, fixed point}
\subjclass[2020]{35R02, 35R06, 35A01, 35A02}

%35R02 (2010-now) PDEs on graphs and networks (ramified or polygonal spaces)

%35R06 (2010-now) PDEs with measure

%35A01 (2010-now) Existence problems for PDEs: global existence, local existence, nonexistence

%35A02 (2010-now) Uniqueness problems for PDEs: global uniqueness, local uniqueness, non-uniqueness

\date{}

\maketitle

%\tableofcontents

\section*{Notation}

For reference, we list some of the most recurrent notation of the paper. 

\subsection*{Measures}
Let $A$ denote a generic set.
\begin{itemize}
  \item $\cB(A)$: Borel subsets of $A$.
  \item $\cM(A)$: Radon measures on $A$.
  \item $\cM^+(A)$: nonnegative Radon measures on $A$. 
%  i.e., $\nu \in \cM^+(A)$ if and only if $\nu$ is Radon and $\nu[\tilde A]\geq 0$ for all $\tilde A \subseteq A$.
  \item Given $\nu\in\cM(\R^d)$ and letting $A\in \cB(\R^d)$, we denote by $\nu^+(A) := \sup_{B\in \cB(A)} \nu(B)$ and $\nu^-(A) := - \inf_{B\in \cB(A)} \nu(B)$ the upper and lower variation measures of $\nu$; the total variation measure of $\nu$ is $|\nu|(A):=\nu^+(A)+\nu^-(A)$ and its total variation norm is $\norm{\nu}_{\TV} := \abs{\nu}(\R^d)$. 
  \item $\cMtv(A)$: Radon measures on $A$ with finite total variation.
  \item $\cMtv^+(A):=\cM^+(A)\cap \cMtv(A)$. %: nonnegative Borel measures on $A$ with finite total variation.
  %\item For any $\rho, \nu\in\cMtv(A)$ we write $\|\rho-\nu\|_{TV}$ the total variation distance.
  \item $\cP(A)$: Borel probability measures on $A$.
  \end{itemize}
  
\subsection*{Graph}  
\begin{itemize}
  \item $\Rddiag := \set{(x,y)\in\R^d\times \R^d : x\ne y}$ is the off-diagonal of $\Rd\times\Rd$.    
  \item $\mu$ sets the underlying geometry of the state space; it belongs to $\cM^+(\R^d)$ and is sometimes referred to as base measure.
  %\item $\rho\in \cP(\R^d)$ denotes a configuration.
  \item $\eta$ is the edge weight function; it maps  $\Rddiag$ to $[0,\infty)$. 
  \item $G$ is the set of edges; i.e., $G= \{ (x,y)\in \Rddiag : \eta(x,y)>0\}$.
  \item $\cV^{\mathrm{as}}(G)$ is the set of antisymmetric vector fields on $G$; that is, $\cV^{\mathrm{as}}(G)= \set{v\: G \to \R : v^\top = - v}$.
  %\item $\dgrad f$ is the nonlocal gradient of a function $f \colon \R^d \to \R$, while $\dgrad \cdot \bj$ is the nonlocal divergence of a measure-valued flux $\bj\in \cM(G)$.
  %; see Definition \ref{def:nl-grad-div}
\end{itemize}

\subsection*{Others} 
\begin{itemize}
  \item $T$ is a positive, finite final time.
  %\item For a curve $\rho:[0,T]\to \cM(A)$, we often write $\rho_t$ for $\rho(t)$. 
  %for the time slice~$t\in [0,T]$.
  \item $\ACt:=\AC([0,T];\cMtv(\R^d))$ is the space of absolutely continuous curves with respect to $\norm{\cdot}_{\TV}$ from $[0,T]$ to $\cMtv(\R^d)$.
  \item Given $a\in\R$, $a_+:=\max\{0,a\}$ and $a_-:=(-a)_+$ are its positive and negative parts, respectively.
  \end{itemize}

\section{Introduction}

In this manuscript, we resume the analysis of Partial Differential Equations (PDEs) on graphs started in our previous work \cite{EsPaScSl-arma}, focusing this time on a larger class of nonlocal continuity equations. The main motivation for this study comes from data science, as graphs represent a relevant ambient space for data representation and classification \cite{Belkin02laplacianeigenmaps, GTS16, GTSspectral, JordanNg, KanVemVet04,RoithBungert22}. However, most of the results obtained so far in the literature are concerned with static problems rather than time-dependent ones.

In \cite{EsPaScSl-arma}, we studied the dynamics driven by nonlocal interaction energies on graphs, whose vertices are the random sample of a given underlying distribution. We interpreted the corresponding PDEs as gradient flows of the nonlocal interaction energies in the space of probability measures, equipped with a quasi-metric obtained from the dynamical transportation cost, following Benamou--Brenier \cite{BB2000}. In the recent papers \cite{HeinzePiSch-1,HeinzePiSch-2}, the analysis is extended to nonlocal cross-interaction systems on graphs with a nonlinear mobility, in the context of nonquadratic Finslerian gradient flows. In \cite{CraigTrillosSlepcev}, dynamics on graphs are shown to be useful for data clustering; indeed, the authors connect the mean shift algorithm with spectral clustering at discrete and continuum levels via Fokker--Planck equations on data graphs.

The study of equations on graphs represents a natural link with the discretization of continuous PDEs, gradient flows, and optimal transport related problems. We start mentioning structure preserving numerical schemes for evolution equations of gradient flow form (see for instance~\cite{CCH15,BailoCarrilloHu2018,CancesGallouetTodeschi2019,Bailo_etal2020,SchlichtingSeis21} and references therein); the use of upwind and similar interpolations showed also beneficial in preserving the second law of thermodynamics, i.e. the entropy decay. 
Inspired by the theory of numerical schemes for local conservation laws, in~\cite{Du2017} a new class of monotonicity-preserving nonlocal nonlinear conservation laws was proposed, in one space dimension. The latter work might be indeed interpreted as an equation on graphs, under some suitable assumptions on the kernel considered. In this regard, it may be interesting to further investigate on the extension of the present manuscript to other nonlocal conservation laws.

Another related question concerns the convergence of discrete optimal transport distances to its continuous counterpart, cf.~\cite{GLADBACH2020204,forkertEvolutionaryGammaConvergence2020,Gladbach2021}. Similarly, the variational convergence of discretization for evolution problems is investigated in~\cite{hraivoronska2022diffusive}. Here the discrete systems obtained can be also seen as special cases of the type of the evolution equations investigated in the current manuscript.
On a different note, we mention~\cite{PeletierRossiSavareTse2020}, where a direct gradient flow formulation of jump processes is recently established --- the authors consider driving energy functional containing entropies. The kinetic relations used there are symmetric, hence excluding for instance the upwind interpolation, which is our main example.

In this work, we consider continuity equations driven by a wide class of velocity fields, including those depending on the the unknown itself, and prove existence and uniqueness of measure-valued, as well as $L^p$-valued, solutions by means of Banach fixed-point theorem. This is a slightly different concept of solution than that used in \cite{EsPaScSl-arma}, where we established a Finslerian gradient flow framework for interaction energies.
As it becomes clear in the following, the geometry of the ambient space influences the analysis and requires novel considerations.

For ease of presentation, we describe the problem first on finite, undirected graphs. Let $X :=\{x_1, \dots, x_n\}\subset\R^d$ be the set of vertices and consider the edge weights $w_{x,y} \geq 0$, satisfying $w_{x,y} = w_{y,x}$ for all $x,y \in X$. For simplicity, we impose that $w_{x,x}=0$. We consider a mass distribution $\rho\colon X \to [0, \infty)$ with $\sum_{x \in X} \rho_x =1$. An example of Ordinary Differential Equations (ODEs) on such a graph preserving the total mass takes the form
\begin{equation}
     \frac{\di \rho_x}{\di t} = - \frac{1}{2} \sum_{y\in X} \bra[\big]{j_{x,y} - j_{y,x}} w_{x,y}, 
     \label{eq:intro:CE}
\end{equation}
The time variation of the mass at a vertex $x$ is triggered by the outgoing and ingoing fluxes, described by the function $j$.
We will be interested in the situation where the flux is obtained by a vector field $v: X\times X \to \R$, along which the mass density $\rho$ is advected. The vector field might itself depend also on the mass density in a local or nonlocal as well as linear or nonlinear way. 
On graphs, the fluxes and velocities $j,v:X\times X\to\R$ are defined on the edges, whereas the mass on the single vertices. 
For this reason, the relation between flux and velocity strongly depends on the chosen mass interpolation on vertex pairs. We consider a general interpolation function $\Phi:\R^3\to\R$ to understand its role for the dynamics better. Hence, the continuity equation in flux-form~\eqref{eq:intro:CE} is complemented by constitutive equation relating the velocity to the flux
\begin{equation*}
      j_{x,y} = \Phi \bra*{\frac{1}{n}\rho_x,\frac{1}{n}\rho_y,v_{x,y}}.
\end{equation*}
In \cite{EsPaScSl-arma}, we also considered the case of graphs with infinite vertices, namely, the PDEs resulting from letting $n$ to $\infty$. Thus, we introduced a unified setup entailing both discrete and continuum interpretations. 

The vertices are points in $\R^d$ and the edges are determined by a nonnegative symmetric weight function $\eta \colon  \Rddiag \to [0, \infty)$; indeed, the set of edges is $G:=\set{ (x,y)\in \Rddiag : \eta(x,y)>0}$, where $\Rddiag = \set{(x,y)\in\R^d\times \R^d : x\ne y}$. From the discrete setting, the set of vertices is replaced by a general measure on $\R^d$, denoted $\mu$; a discrete graph with vertices $X:=\{x_1, \dots, x_n\} \subset \R^d$ corresponds then to $\mu$ being the empirical measure of $X$, i.e., $\mu = \frac{1}{n} \sum_{i=1}^n \delta_{x_i}$. This generalization is natural in applications to machine learning, since data have the form of a point cloud randomly sampled from some measure in Euclidean space. With this notation, the PDEs we study have the form
\begin{subequations}
\label{eq:intro-CE-PDE}
\begin{align}
    \partial_t\rho + \dgrad\cdot\bj &= 0 ,\\
    \bj&=F^\Phi(\mu;\rho,v),
\end{align}
\end{subequations}
where $\dgrad$ and $\dgrad\cdot$ are the nonlocal gradient and divergence, respectively (cf. Definition~\ref{def:nl-grad-div} below), and $F^\Phi$ is an interpolation-dependent flux.

In \cite{EsPaScSl-arma}, we considered the upwind interpolation between vertices, as it is a reasonable choice for both the dynamics and the gradient flow structure. More precisely, we fixed $\Phi(a,b,v)=av_+-bv_-$ and introduced the following nonlocal continuity equation:
\begin{equation*}
    \partial_t\rho_t(x)+\int_\Rd\left(\rho_t(x)v_t(x,y)_+-\rho_t(y)v_t(x,y)_-\right)\eta(x,y)\di \mu(y) = 0,
\end{equation*}
for $\mu\text{-a.e. } x \in \Rd$. Note that we let here $\rho\ll\mu$ for ease of presentation, although it is not necessary. We focused on the specific case of the nonlocal-interaction equation, that is,
\begin{align}
\label{eq:nlnl-intro}
\begin{split}
    \partial_t\rho_t(x) & = - \int_\Rd j_t(x,y) \eta(x,y) \di\mu(y) =: - (\dgrad \cdot j_t) (x),  \\
    j_t(x,y) &= \rho_t(x) v_t(x,y)_+ -  \rho_t(y) v_t(x,y)_-, \\
    v_t(x,y) &= - \left(K*\rho_t(y) - K*\rho_t(x) + P(y) - P(x)\right).
\end{split} \tag{NL$^2$IE}
\end{align}
The equation above is actually a particular case of a nonlocal conservation law, as the velocity field depends on the configuration itself. The theory of generalised Wasserstein gradient flows was shown to be useful to prove existence of weak solutions to \eqref{eq:nlnl-intro} and to provide information on the underlying geometry structure of the configuration space, which is the set of probability measures with finite second-order moments. The latter, equipped with quasi-metric introduced in \cite{EsPaScSl-arma}, has Finsler structure, rather than Riemannian. Among others, open problems include the contractivity of the quasi-distance (cf.~\cite{OhtaSturm09,OhtaSturm12}), the stability and uniqueness of weak solutions.

Based on the above considerations, in this paper, we obtain existence and uniqueness of measure and $L^p$ solutions for the class of PDEs \eqref{eq:intro-CE-PDE} by means of a classical Banach fixed-point argument. This complements the analysis started in \cite{EsPaScSl-arma}, as it concerns general flux interpolations as well as a larger class of velocity fields. The structure of the graph influences the analysis of the equations in this setting. Indeed, some analogous properties in the Euclidean case are not easily derived, depending on the interpolation chosen. Therefore, as a byproduct of our study, we provide properties of the dynamics in relation to the interpolation considered, such as positivity preservation and $L^p$ regularity. To the best of our knowledge this is the first result in these directions.
\smallskip

The paper is structured as follows. We introduce preliminary notions in Section~\ref{sec:results} to explain the setup. Section~\ref{sec:nonl-cont-equat} is devoted to the Nonlocal Continuity Equation (NCE) and emphasizes the fundamental role of the flux interpolation. From there, we prove basic properties of the NCE, highlighting analogies with and differences from the more standard Euclidean setting. In Section~\ref{sec:nonl-nonl-inter}, we prove the main result of the manuscript, namely, the existence and uniqueness of measure solutions for the NCE. We include velocity fields depending on the solution itself, in which case we also refer to the NCE as a Nonlocal Conservation Law (NCL). Section~\ref{sec:lp} is focused on $L^p$ solutions and positivity preservation, only proven for the upwind interpolation.

\section{Setup}
\label{sec:results}

\subsubsection*{Nonlocal graph structure}

Let us fix a measure $\mu \in \cM^+(\R^d)$ and a measurable function $\eta\: \Rddiag \to [0,\infty)$, and set $G := \{(x,y)\in \Rddiag : \eta(x,y)>0\}$. We always assume following:
\begin{namedthm}{($\boldsymbol{\eta}$)}
  \namedlabel{hyp:eta}{{\bfseries ($\boldsymbol{\eta}$)}}
  $\eta$ is continuous, bounded and symmetric on $G$.
\end{namedthm}
\noindent We often refer to $\mu$ as the \emph{base measure} and to $\eta$ as the \emph{weight function}. In this sense, $(\mu,\eta)$ defines a, possibly uncountable, weighted, undirected graph. A finite graph would correspond to the base measure $\mu_n=\frac{1}{n}\sum_i\delta_{x_i}$ for a set of points $\{x_1,x_2,\dots,x_n\}$.

\subsubsection*{Total variation distance}

For two measures $\rho^1,\rho^2\in \cMtv(\R^d)$, we define their total variation distance by
\begin{equation*}
  %\dtv(\rho^0,\rho^1) = 
  \norm{\rho^1-\rho^2}_{\TV} = 2\sup_{A\in \cB(\R^d)} \abs*{\rho^1[A] - \rho^2[A]}.
\end{equation*}
The factor $2$ is present only for convenience since we restrict to measures with finite and equal total variation, so that $\norm{\rho^1- \rho^2}_{\TV}=\abs*{\rho^1-\rho^2}(\R^d)$. We equip the sets $\cMtv(\Rd)$ and $\cP(\Rd)$ with the total variation distance.

\subsubsection*{Gradients and divergences}

We recall here the notions of nonlocal gradient and divergence on $G$.

\begin{definition}[Nonlocal gradient and divergence] 
  \label{def:nl-grad-div}
  For any $\phi \: \Rd \to \R$, we define its \emph{nonlocal gradient} $\dgrad \phi \colon G \to \R$ by
  \begin{equation*} 
    \dgrad\phi(x,y)=\phi(y)-\phi(x) \quad \mbox{for all $(x,y)\in G$}.
  \end{equation*}
  For any Radon measure $\bj\in \cM(G)$, its \emph{nonlocal divergence} $\dgrad\cdot\bj \in \cM(\R^d)$ is defined as the adjoint of $\dgrad$ with respect to $\eta$, i.e., for any $\phi \: \Rd \to \R$ continuous and vanishing at infinity, there holds
  \begin{equation*}
    \begin{aligned}
      \int_\Rd \phi \dx\dgrad \cdot \bj &= - \frac12\iint_G\dgrad\phi(x,y) \eta(x,y)\di\bj(x,y)\\
      &= \frac{1}{2}\int_{\Rd} \phi(x) \int_{\Rd\setminus\{x\}} \eta(x,y) \bra*{ \di\bj(x,y) - \di\bj(y,x)}.
    \end{aligned}
  \end{equation*}
  In particular, for $\bj$ antisymmetric, that is, $\bj\in \cM(G)$ and $\bj^\top = - \bj$, denoted $\bj\in\cM^{\mathrm{as}}(G)$, we have
  \begin{equation*} 
    \int_\Rd \phi \dx\dgrad\cdot\bj= \iint_G \phi(x) \eta(x,y) \di\bj(x,y) .
  \end{equation*}
\end{definition}
With this notion of divergence, we can consider a nonlocal continuity equation (cf. Definition \ref{def:nce-flux-form} below) defined on a suitable subclass of absolutely continuous curves denoted by $\ACt=\AC([0,T];\cMtv(\R^d))$. More precisely, $\ACt$ is the set of curves from $[0,T]$ to $\cMtv(\Rd)$
%which are absolutely continuous with respect to $\dtv$, that is, the set of curves $\rho\colon [0,T] \to \cMtv(\Rd)$
such that there exists $m\in L^1([0,T])$ with
\begin{equation*}
  \norm{\rho_s-\rho_t}_{\TV} \leq \int_s^t m(r) \dx{r}, \qquad \text{for all } 0\leq s< t\leq T. 
\end{equation*}

\section{Nonlocal Continuity Equation (NCE)}
\label{sec:nonl-cont-equat}

In this section, we study the nonlocal continuity equation on the graph defined by $(\mu,\eta)$. 
%We focus on general structures given by suitable flux interpolations.
First, we define the concept of measure-valued solution.

\begin{definition}[Measure-valued solution for the NCE]
  \label{def:nce-flux-form}
  A measurable pair $(\rho,\bj)\colon [0,T] \to \cMtv(\Rd)\times \cM(G)$ is a \emph{measure-valued} (or simply \emph{measure}) solution to the NCE, denoted as 
  \begin{equation}\tag{NCE}
    \label{eq:nce-expr}
    \partial_t \rho + \dgrad \cdot\bj = 0,
  \end{equation}
  provided that, for any $A\in \cB(\Rd)$, it holds that
  \begin{listi}
    \item $\rho \in \ACt$;
    \item $(\bj_t)_{t\in[0,T]}$ is Borel measurable and $\bra*{t\mapsto \dgrad\cdot\bj_t[A]} \in L^1([0,T])$;
    \item $(\rho,\bj)$ satisfies,
    \begin{equation*}
      \rho_t[A] +  \int_{0}^t \dgrad \cdot \bj_s[A] \di s = \rho_{0}[A]  \qquad \text{for a.e.~$t \in [0,T]$};
    \end{equation*}
  \end{listi}
in this case, we write $(\rho,\bj)\in\CE([0,T])$.
\end{definition}
In the above definition, the absolute continuity of a measure solution $\rho$ is ensured by the integrability of the flux divergence.
Moreover, $\rho$ does not need to be nonnegative, i.e., so that $\rho_t\geq0$ for a.e. $t\in[0,T]$, for the definition to make sense; in fact, positivity preservation is analyzed in Section~\ref{sec:lp}.

\subsection{Flux interpolations}

We provide a class of flux interpolations generalizing our work in~\cite{EsPaScSl-arma}, where we only studied the upwind interpolation. We consider a minimal set of assumptions on the interpolation to achieve well-posedness.
\begin{definition}[Admissible flux interpolation]
  \label{def:AdmFlux}
  A measurable function $\Phi\colon\R^3 \to \R$ is called an \emph{admissible flux interpolation} provided that the following conditions hold:
  \begin{listi}
    \item $\Phi$ satisfies 
    \begin{equation}\label{ass:normalisation_phi}
        \Phi(0,0;v)=\Phi(a,b;0)=0,\quad \text{ for all } a,b,v\in\R; 
    \end{equation}
    \item $\Phi$ is \emph{Lipschitz} in its arguments in the sense that, for some $L_{\Phi}>0$, any $a,b,c,d,v,w\in\R$, it holds
	 \begin{subequations}\label{eq:Phi-cond}
		\begin{align}
			\abs*{\Phi(a,b;w) - \Phi(a,b;v)} &\leq L_{\Phi} (|a|+|b|) \abs*{w-v}; \label{eq:Phi-cond-1}\\
			\abs*{\Phi(a,b;v) - \Phi(c,d;v)} &\leq L_{\Phi} (\abs*{a-c} + \abs*{b-d}) \abs*{v}\label{eq:Phi-cond-2};
		\end{align}
	\end{subequations}
    \item $\Phi$ is \emph{positively one-homogeneous} in its first and second arguments, that is, for all $\alpha>0$ and $(a,b,w)\in \R^3$, it holds 
    \begin{equation*}
      \Phi(\alpha a,\alpha b; w) = \alpha \Phi(a,b;w).
    \end{equation*}
    \end{listi}
\end{definition}

\begin{exm}
  Here follow examples of admissible flux interpolations $\Phi$ according to Definition \ref{def:AdmFlux}.
  \begin{itemize}
    \item \emph{Upwind interpolation.} One important case 
    is given by the \emph{upwind} interpolation $\Phi_{\mathrm{upwind}}$ defined as
    \begin{equation} 
      \label{eq:Phi:upwind}
      \Phi_{\mathrm{upwind}}(a,b;w) = a w_+ - b w_- \qquad\text{for } (a,b,w)\in \R^3.
    \end{equation}
    \item \emph{Mean multipliers.} Another case is \emph{product} interpolation $\Phi_{\mathrm{prod}}$, which is of the form
     \begin{equation*}
       \Phi_{\mathrm{prod}}(a,b;w) = \phi(a,b) w \qquad\text{for } (a,b,w)\in \R^3,
     \end{equation*}
     with $\phi\:\R^2\to\R$ any measurable function satisfying, for some $L_\Phi>0$, 
     \begin{align*}
         &|\phi(a,b)|\leq L_\Phi\max\set*{|a|,|b|}, \\
         &|\phi(a,b)-\phi(c,d)|\le L_\Phi(|a-c|+|b-d|),\\
         &\phi(\alpha a, \alpha b)= \alpha \phi(a,b),\\
         &\phi(a,b)=\phi(b,a),
     \end{align*}
    for all $\alpha \geq 0$ and $a,b,c,d\in\R$. Common choices for $\phi$ are as below:
    \begin{itemize}
       \item \textit{Arithmetic mean.} $\phi_{\mathrm{AM}}(a,b):= \frac{a+b}{2}$;
    %   \item \textit{Geometric mean.} $\phi_{\mathrm{GM}}(a,b) = \sqrt{ab}$;
    %   \item \textit{Logarithmic mean.} $\phi_{\mathrm{LM}}(a,b)=  \int_0^1 a^{1-s} b^s \dx{s}$.
        \item \textit{Minimal mean.} $\phi_{\mathrm{min}}(a,b):=\min\{a,b\};$
        \item \textit{Maximal mean.} $\phi_{\mathrm{max}}(a,b):=\max\{a,b\}.$
     \end{itemize}
    We note that some common choices, such as the geometric mean and the logarithmic mean, do not satisfy the Lipschitz condition stated above, which is essential for the fixed-point argument we use later to establish well-posedness. This situation may be remedied by a suitable Lipschitz regularization of those examples, although we do not explore this possibility in the present paper.
  \end{itemize}
\end{exm}

\begin{definition}[Admissible flux]\label{def:FPhi}
  Let $\Phi$ be an admissible flux interpolation, and let $\rho\in\cMtv(\Rd)$ and $w\in\cV^{\mathrm{as}}(G) := \set{v\: G \to \R : v^\top = - v}$. Furthermore, take $\lambda\in\cM^+(\Rdd)$ such that $\rho\otimes\mu,\mu\otimes\rho\ll \lambda$ (e.g., $\lambda = \abs{\rho}\otimes \mu + \mu \otimes \abs{\rho}$). Then, the \emph{admissible flux} $F^\Phi[\mu;\rho,w]\in\cM(G)$ at $(\rho,w)$ is defined by
  \begin{equation}
    \label{eq:def:FPhi}
    \dx F^\Phi[\mu;\rho,w] = \Phi\bra*{\pderiv{(\rho \otimes \mu)}{\lambda},\pderiv{(\mu\otimes\rho)}{\lambda}; w } \dx \lambda.
  \end{equation} 
  \end{definition}
 Note that because of the one-homogeneity of $\Phi$, the expression in \eqref{eq:def:FPhi} is independent of the choice of $\lambda$. The nonlocal continuity equation of Definition~\ref{def:nce-flux-form} with the notation of Definition~\ref{def:FPhi} reads
\begin{equation}
    \label{eq:cont-eq-strong}
    \partial_t \rho + \dgrad \cdot F^\Phi[\mu;\rho_t,v_t] = 0\tag{NCE},
  \end{equation}
with integral form, for all $A\in \cB(\R^d)$, given by
    \begin{equation}
     \label{eq:strong-sol}
     \rho_t[A] +  \int_{0}^t \dgrad \cdot F^\Phi[\mu;\rho_s,v_s][A] \di s = \rho_{0}[A],  \qquad \text{for a.e. $t \in [0,T]$}.
    \end{equation}

\subsection{Basic properties}

We highlight some properties of~\eqref{eq:cont-eq-strong} analogous to those in Euclidean setting, though intrinsically different due to the underlying graph structure. The well-posedness is treated in Section~\ref{sec:nonl-nonl-inter}, where we consider a more general scenario, in particular including~\eqref{eq:cont-eq-strong}. 

% We start by a result with no sign constraint on the unknown $\rho$.
\begin{proposition}[Integrability, support and mass preservation for the NCE]
  \label{prop:well-posed}
  Let $\rho_0\in\cMtv(\R^d)$ and let $v\:[0,T] \to \cV^{\mathrm{as}}(G)$ satisfy, for some $C_v>0$,
  \begin{equation}
    \label{eq:basic-strong:eta:v}
     \int_0^T \sup_{x\in\Rd}\int_{\Rd\setminus\{x\}} |v_t(x,y)| \eta(x,y) \di \mu(y) \leq C_v .
  \end{equation} 
Let also $\Phi$ be an admissible flux interpolation and $\rho: [0,T]\to \cMtv(\R^d)$ be such that~\eqref{eq:strong-sol} is satisfied. Then, the following properties hold:
%such that $(\rho,v)$ is a strong solution to~\eqref{eq:cont-eq-strong}, 
\begin{itemize}
	\item $t \mapsto \dgrad \cdot F^\Phi[\rho_t,v_t][A] \in L^1([0,T])$ \emph{(flux integrability)};
    \item $\rho\in L^\infty([0,T];\cMtv(\R^d))$ \emph{(time boundedness)};
    \item $\rho_t[\R^d]=\rho_0[\R^d]$ for all $t\in [0,T]$ \emph{(mass preservation)};
    \item $\rho\in \ACt$ \emph{(absolute continuity)};
    \item if $\supp \rho_0\subseteq \supp \mu$, then $\supp \rho_t \subseteq \supp \rho_0$ for a.e. $t\in [0,T]$ \emph{(support inclusion)}. 
\end{itemize}
%Moreover, under condition \eqref{eq:basic-strong:eta:v}, any solution $(\rho,v)$ is mass preserving, that is, $\rho_t[\R^d]=\rho_0[\R^d]$ for all $t\in [0,T]$.
\end{proposition}
\begin{proof}
We split the proof according to each item above.

\medskip\noindent
\textit{Flux integrability---}For all $A\in\cB(\Rd)$ and $t\in[0,T]$, we have
  \begin{align}
    \dgrad\cdot F^\Phi[\mu;\rho_t,v_t][A]&=-\frac{1}{2}\iint_G\dgrad\chi_A(x,y)\eta(x,y)\di F^\Phi[\mu;\rho_t,v_t](x,y) \nonumber\\
    &=-\frac{1}{2}\iint_G\dgrad\chi_A\Phi\bra*{\pderiv{(\rho_t\otimes \mu)}{\lambda},\pderiv{(\mu\otimes\rho_t)}{\lambda}; v_t }\eta \dx \lambda. \nonumber
    %&= \iint_{G \cap A \times \R^d} \Phi\bra*{\pderiv{(\rho_t\otimes \mu)}{\lambda},\pderiv{(\mu\otimes\rho_t)}{\lambda}; v_t }\eta \dx \lambda , \label{eq:flux:explicit}
  \end{align}
%where we used the antisymmetry of $v$ and $\Phi$ from~\eqref{eq:Phi-antisymm}.
Next, using~\eqref{ass:normalisation_phi} and \eqref{eq:Phi-cond-1} with $w=0$, symmetry of $\eta$, antisymmetry of $v$, and~\eqref{eq:basic-strong:eta:v}, we estimate, for any $t\in [0,T]$, that
  \begin{align}
    \int_0^t  \abs*{\dgrad\cdot F^\Phi[\mu;\rho_s,v_s][A]} \di s  &\le \frac{L_\Phi}{2}\! \int_0^t \iint_G |v_s| \eta\left(\pderiv{|\rho_s|\otimes \mu}{\lambda}\!+\!\pderiv{\mu\otimes |\rho_s|}{\lambda}\right) \dx \lambda \dx s \nonumber\\
    &\le L_\Phi\int_0^t \iint_G|v_s(x,y)|\eta(x,y)\di\mu(y)\di|\rho_s|(x)  \dx{s} \nonumber\\
    &\leq L_{\Phi}\int_0^t \overline{v}_s \, |\rho_s|[\Rd]\di s , \label{eq:intergrability_flux:p0}
  \end{align}
  where $\overline{v}_s := \sup_{x\in \R^d} \int_{\Rd\setminus\{x\}} |v_t(x,y)| \eta(x,y) \di \mu(y)$.

\medskip\noindent
\textit{Time boundedness---} For a.e. $t\in[0,T]$, the integral form~\eqref{eq:strong-sol} entails
  \[
    \abs{\rho_t}[\R^d] \leq \abs{\rho_0}[\R^d] + L_{\Phi}\int_0^t \overline{v}_s |\rho_s|[\Rd] \di s .
  \]
  Then, Gronwall's inequality provides, for a.e. $t\in [0,T]$, the a priori bound $\abs{\rho_t}[\R^d]\leq \abs{\rho_0}[\R^d] e^{L_\Phi C_v} < \infty$. Hence $\rho\in L^\infty([0,T];\cMtv(\R^d))$.
  
  \medskip\noindent
  \textit{Mass preservation---}This is a simple consequence of $\dgrad\chi_{\R^d} = 0$, which yields $\dgrad\cdot F^\Phi[\rho_t,v_t][\R^d]=0$ for all $t\in[0,T]$. Hence \eqref{eq:strong-sol} implies that $\rho$ is mass preserving. We also infer the integrability of the flux from~\eqref{eq:intergrability_flux:p0}.

\medskip\noindent
  \textit{Absolute continuity---}For any $A\in \cB(\R^d)$, we have  $t\mapsto |\dgrad \cdot F^\Phi[\mu;\rho_t,v_t][A]|$ belongs to $L^1([0,T])$. Hence $\rho\in \ACt$.
  
  \medskip\noindent
   \textit{Support inclusion---}Note that for $A=\Rd\setminus\supp\mu$ and a.e. $t\in[0,T]$, the solution $\rho$ satisfies
\begin{align*}
    \rho_t[A]=\rho^0[A]&-\frac{1}{2}\int_0^t\iint_{G\cap A\times\supp\mu}\Phi\bra*{\pderiv{(\rho_s\otimes \mu)}{\lambda},0; v_s}\di \lambda(x,y)\di s \\
    &+\frac{1}{2}\int_0^t\iint_{G\cap \supp\mu\times A}\Phi\bra*{0,\pderiv{(\mu\otimes \rho_s)}{\lambda}; v_s}\di \lambda(x,y)\di s;
\end{align*}
thus, we get the estimate
\begin{align*}
    \abs*{\rho_t}[A] &\leq \abs*{\rho_0}[A] + L_\Phi\int_0^t \abs*{\rho_s}[A] \overline{v}_s \dx{s},
\end{align*}
and, by Gronwall's inequality, we also get $\abs*{\rho_t}[A]\leq e^{C_V L_\Phi} \abs*{\rho_0}[A]$. We conclude by noting that $\abs*{\rho_0}[A]=0$ by assumption.
%Choose now $\rho\in \ACt$ such that the iterative expression~\eqref{eq:strong-sol} holds \FP{[fixed point - refer to the nonlocal below]}; then $\rho$ belongs in fact to $\AC([0,T];\cMtv(\Rd))$ and satisfies that $(\rho,v)$ is a strong solution to~\eqref{eq:cont-eq-strong}.
\end{proof}
\begin{rem}
  Condition~\eqref{eq:basic-strong:eta:v} is the analogue of the weak-compressibility assumption classically used for the continuity equation $\partial_t \rho_t + \nabla\cdot \bra*{v_t \rho_t}=0$, with vector field $v:[0,T]\times \R^d \to \R^d$ (see, e.g.,~\cite{DiPernaLions1989, Ambrosio_inv2004}).
  More precisely, in the Euclidean setting, the assumption in~\eqref{eq:basic-strong:eta:v} takes the form $\nabla \cdot v \in L^1([0,T];L^\infty(\R^d))$ and is used to control of $\norm{\rho}_{L^\infty([0,T];L^p(\R^d))}$, for any $p\in [1,\infty)$ (cf.~\cite[Prop~II.1.]{DiPernaLions1989}).
  In our setting, the structural properties of the graph, encoded in $(\mu,\eta)$ and the flux interpolation $\Phi$, require a refined analysis involving a careful regularization argument when treating $L^p$ solutions; we refer the reader to Section~\ref{sec:lp}, where those questions are studied for solutions possessing a density.
\end{rem}

\section{Nonlocal Conservation Law (NCL)}
\label{sec:nonl-nonl-inter}

%In Section \ref{sec:nonl-cont-equat} we only deal with properties that the nonlocal continuity equation satisfies, based on those in the Euclidean space $\Rd$. 
We focus here on the general case where the velocity field depends on the solution itself. More precisely, we provide well-posedness to~\eqref{eq:cont-eq-strong} for a vector field of the form
\begin{equation*}
  v_t(x,y) = V_t[\rho_t](x,y) \qquad \text{for all $t\in[0,T]$},
\end{equation*}
for some $V\: [0,T] \times \cMtv(\Rd) \to \cV^{\mathrm{as}}(G)$. For the reader's convenience we write the following straightforward generalization of Definition~\ref{def:nce-flux-form} to what we refer to as \emph{Nonlocal Conservation Law} (NCL).
\begin{definition}[Measure-valued solution to the NCL] 
  \label{def:strongNCL} 
  Given an admissible flux interpolation $\Phi$ and a measurable map $V\: [0,T] \times \cMtv(\Rd) \to \cV^{\mathrm{as}}(G)$, a curve $\rho \: [0,T] \to \cMtv(\Rd)$ is said to be a \emph{measure-valued} (or simply \emph{measure}) solution to the NCL, denoted as
  \begin{equation}
    \label{eq:ncl-eq-strong}
    \partial_t \rho + \dgrad \cdot F^\Phi[\mu;\rho,V(\rho)] = 0,\tag{NCL}
  \end{equation}
  provided that, for any $A\in \cB(\Rd)$, it holds that
  \begin{listi}
    \item $\rho\in \ACt$;
    \item $t \mapsto \dgrad \cdot F^\Phi[\mu;\rho_t,V_t(\rho_t)][A] \in L^1([0,T])$;
    \item $\rho$ satisfies
    \begin{equation}
      \label{eq:ncl-strong-sol}
      \rho_t[A] +  \int_{0}^t \dgrad \cdot F^\Phi[\mu;\rho_s,V_s(\rho_s)][A] \di s = \rho_{0}[A]  \qquad \text{for a.e. $t \in [0,T]$}.
    \end{equation}
  \end{listi} 
\end{definition}
\begin{exm}\label{ex:nlnlie}
  An important example of a map $V$ in Definition \ref{def:strongNCL} is that stemming from the convolution with an interaction \emph{kernel} (or \emph{potential}) $K\:\Rd\times\Rd\to\R$, which yields the \emph{Nonlocal Nonlocal Interaction Equation} (NL$^2$IE), to which we can add an external potential $P\:\Rd\to\R$. Namely, in this case, for $\rho\:[0,T] \to \cMtv(\Rd)$, $t\in [0,T]$ and $(x,y)\in G$, the vector field $V$ is given by
\begin{equation*}
  V_t[\rho_t](x,y) = -\dgrad (K*\rho_t)(x,y) - \dgrad P(x,y) .
\end{equation*}
When the interpolation is chosen to be the upwind one \eqref{eq:Phi:upwind}, we get the equation studied in the optimal-transport, weak-measure setting of~\cite{EsPaScSl-arma}.
\end{exm}

%\subsection{Well-posedness}
%\label{sec:well-posedness-nce}

Our well-posedness proof of~\eqref{eq:ncl-eq-strong}, and thus \eqref{eq:cont-eq-strong}, is based on a fixed-point argument and only applies to measures with fixed total variation, which is consistent with the mass-preservation property from Proposition \ref{prop:well-posed}. For all $M>0$, we introduce the notation
\begin{align*}
  \ACt^M = \AC([0,T];\cMtv^M(\Rd)), \qquad \cMtv^M(\Rd) = \set*{\rho\in\cMtv : \abs*{\rho}[\Rd] = M}.
\end{align*}
Note that, for any $\rho^0,\rho^1\in\cMtv^M(\Rd)$, we have the identity
\begin{equation*}
  \norm{\rho^1-\rho^2}_{\TV} = 2\sup_{A\in \cB(\R^d)} \abs*{\rho^1[A] - \rho^2[A]}=|\rho^1-\rho^2|(\Rd).
\end{equation*}
% \AE{\begin{align*}
%   \ACt^M = \AC([0,T];L^1_M(\Rd)), \quad L^1_M(\Rd) = \set*{\rho\in\L^1_\mu(\Rd) : \int_\Rd|\rho(x)|d\mu(x) = M}.
% \end{align*}}
Throughout this section we fix $M\geq0$, $\rho^0\in\cMtv^M(\Rd)$ and $\Phi$ an admissible flux interpolation (cf. Definition \ref{def:AdmFlux}). With any $V\colon [0,T] \times \cMtv^M(\Rd) \to \cV^{\mathrm{as}}(G)$ such that, for some $C_V>0$,
\begin{equation}\label{ass:compressibility:V}
  \sup_{t\in [0,T]} \sup_{\rho\in\cMtv^M(\Rd)} \sup_{x\in \Rd}  \int_{\Rd\setminus\{x\}} |V_t[\rho](x,y)| \eta(x,y) \di \mu(y)\dx t \leq C_V,
\end{equation}
we associate the \emph{solution map} $S_T^V\colon \ACt^M \to \ACt^M$, defined, for $t \in [0,T]$ and $A\in \cB(\Rd)$, by
\begin{equation*}
  S_T^V(\rho)(t)[A] := \rho^0[A] - \int_0^t \dgrad \cdot F^\Phi[\mu;\rho_s,V_s(\rho_s)][A]\di s .
\end{equation*}
Note that \eqref{ass:compressibility:V} is an $L^\infty(L^\infty)$-type of bound for the nonlocal divergence; it is thus slightly stronger than the similar~\eqref{eq:basic-strong:eta:v} of $L^1(L^\infty)$-type under which we have boundedness of solutions in Proposition~\ref{prop:well-posed}.

We establish well-posedness under a Lipschitz assumption on $\rho\mapsto V[\rho]$ on the space $\ACt$, which we endow with the distance $\dT$ defined by
\begin{equation*}
  \dT(\rho,\sigma) \!=\! \norm{\rho-\sigma}_{L^\infty([0,T];\cMtv(\R^d))} \!=\! \sup_{t\in [0,T]} \norm{\rho_t-\sigma_t}_{\TV} \quad \text{for all $\rho,\sigma\in\ACt$}.
\end{equation*}
\begin{lemma}\label{lem:S-Lip}
  Let $V\colon [0,T] \times \cMtv^M(\Rd) \to \cV^{\mathrm{as}}(G)$ satisfy the uniform-compressibility assumption~\eqref{ass:compressibility:V} for some $C_V\in (0,\infty)$ and suppose that there exists a constant $L_V\ge0$ such that, for all $t\in [0,T]$ and all $\rho,\sigma \in \cMtv^M(\Rd)$,
  \begin{equation}
    \label{ass:Lipschitz:v}
      \sup_{x\in \Rd}\int_{\Rd\setminus\{x\}} |V_t[\rho](x,y) - V_t[\sigma](x,y)| \eta(x,y) \di\mu(y)\di t \leq L_V \norm{\rho - \sigma}_{\TV}. 
  \end{equation}
Then, for all $\rho,\sigma \in \ACt^M$, the contraction estimate
  \begin{equation*}
    \dT(S_T^V(\rho),S_T^V(\sigma)) \leq \alpha T \dT(\rho,\sigma) ,
  \end{equation*}
  holds  for $\alpha:= L_{\Phi}\bra*{M L_V + C_V}$, where $L_\Phi$ is as in \eqref{eq:Phi-cond}.
  
  In particular, for $T>0$ such that $T<1/\alpha$, there exists a unique measure solution $\rho$ to \eqref{eq:ncl-eq-strong} on $[0,T]$ such that $\rho_{0} = \rho^0$.
\end{lemma}
\begin{proof}
  Let $\rho,\sigma \in \ACt^M$ and let $t \in [0,T]$. We rewrite, for $s\in[0,T]$,
  \begin{equation}\label{eq:diff-flux}
    \dgrad \cdot F^\Phi[\mu;\rho_s,V_s(\rho_s)][A] - \dgrad \cdot F^\Phi[\mu;\sigma_s,V_s(\sigma_s)][A]= \I_s + \II_s ,
  \end{equation}
 where
  \begin{align*}
      \I_s &= \frac{1}{2}\iint_{G}\dgrad\chi_A(x,y) \Biggl[ \Phi\bra*{\pderiv{(\sigma_s\otimes \mu)}{\lambda},\pderiv{(\mu\otimes\sigma_s)}{\lambda}; V_s[\sigma_s]} \\
    &\qquad\qquad\qquad - \Phi\bra*{\pderiv{(\sigma_s\otimes \mu)}{\lambda},\pderiv{(\mu\otimes\sigma_s)}{\lambda}; V_s[\rho_s]}\Biggr] \eta \dx \lambda , \\
    \II_s &= \frac{1}{2}\iint_{G}\dgrad\chi_A(x,y) \Biggl[
    \Phi\bra*{\pderiv{(\sigma_s\otimes \mu)}{\lambda},\pderiv{(\mu\otimes\sigma_s)}{\lambda}; V_s[\rho_s]}\\
    &\qquad\qquad\qquad  -\Phi\bra*{\pderiv{(\rho_s\otimes \mu)}{\lambda},\pderiv{(\mu\otimes\rho_s)}{\lambda}; V_s[\rho_s]}\Biggr] \eta \dx \lambda .
  \end{align*}
  
  For the fist term, we apply the Lipschitz assumptions~\eqref{eq:Phi-cond-1} on $\Phi$ and~\eqref{ass:Lipschitz:v} on $V$, and use the antisymmetry of $V_t(\rho_t)$ and $V_t(\sigma_t)$ and the symmetry of $\eta$ (cf.~\ref{hyp:eta}) to obtain
  \begin{align*}
      \int_0^t \abs{\I_s}\dx{s} &\leq \frac{L_\Phi}{2}\int_0^t\iint_{G} \left|V_s[\sigma_s] - V_t[\rho_s] \right| \eta  \left( \dx(\abs*{\sigma_s}\otimes \mu) + \dx(\mu\otimes\abs*{\sigma_s}) \right)\dx s \\
      &\leq L_\Phi \int_0^t \iint_{G} \abs[\big]{V_s[\sigma_s] - V_t[\rho_s] }(x,y) \eta(x,y)  \dx(\abs*{\sigma_s}\otimes \mu)(x,y)\di s \\
      &\leq\! L_\Phi \!\sup_{s\in [0,t]} \abs{\sigma_s}[\R^d]\! \int_0^t \!\sup_{x\in \R^d} \int_{\R^d\setminus\set{x}} \abs[\big]{V_s[\sigma_s] \!-\! V_t[\rho_s]}(x,y) \eta(x,y)  \!\dx\mu(y) \!\dx s\\
      &\leq L_\Phi L_V M T \, \dT(\rho,\sigma) . 
  \end{align*}

As for $\II_s$, we use the Lipschitz assumption~\eqref{eq:Phi-cond-2} on $\Phi$, again the antisymmetry of $V_t(\rho_t)$ and the symmetry of $\eta$ (recall~\ref{hyp:eta}), and apply the compressibility of $V$ given in~\eqref{ass:compressibility:V} to get
\begin{align*}
    \int_0^t\abs{\II_s}\di s&\le \frac{L_\Phi}{2}\int_0^t\iint_{G}  \abs*{V_s[\rho_s]}(x,y) \left( \left|\pderiv{(\sigma_s\otimes \mu)}{\lambda}(x,y) - \pderiv{(\rho_s\otimes \mu)}{\lambda}(x,y)\right| \right.\\
    &\qquad\qquad + \left. \left| \pderiv{(\mu\otimes\sigma_s)}{\lambda}(x,y) - \pderiv{(\mu\otimes\rho_s)}{\lambda}(x,y)\right| \right) \eta(x,y) \dx \lambda(x,y)\dx s\\
    &\leq \frac{L_\Phi}{2}\int_0^t\iint_{G}\abs*{V_s[\rho_s]}\eta \bra*{\di(\abs{\sigma_s-\rho_s}\otimes \mu)+\di(\mu\otimes\abs{\sigma_s-\rho_s})}\di s\\
    &\leq  L_\Phi C_VT \dT(\rho,\sigma). 
    %&\le 2 L_\Phi \int_0^T\bar{v}_s|\sigma_s-\rho_s|(\R^d)\di s ,
    %&\le L_\Phi \int_0^T\bar{v}_s|\sigma_s-\rho_s|(A)\di s+2ML_\Phi\int_0^T\bar{v}_s\di s,
\end{align*}
All in all, taking the suprema over Borel sets and over time in \eqref{eq:diff-flux} gives
  \begin{align*}
    \dT\bra*{S_T^V(\rho), S_T^V(\sigma)}
    \leq  L_{\Phi}\bra*{M L_V + C_V}T \dT(\rho,\sigma)=: \alpha T\dT(\rho,\sigma). 
  \end{align*}
  
  The existence and uniqueness when $T<1/\alpha$ is a direct consequence of the Banach fixed-point theorem in the metric space $\mathcal{AC}_{T}^M$ applied to $S_{T}^V$.
\end{proof}

\begin{rem}
For \eqref{eq:nce-expr}, one has to control only the term $\II_s$, and so the condition in \eqref{ass:compressibility:V} is enough to get the contraction estimate and well-posedness.
\end{rem}

\begin{theorem}[Well-posedness for \eqref{eq:ncl-eq-strong}]\label{thm:existence-strong}
  Let $V\colon [0,T] \times \cMtv^M(\Rd) \to \cV^{\mathrm{as}}(G)$ and suppose there are constants $C_V,L_V>0$ so that, for all $t\in[0,T]$ and all $\rho,\sigma \in \cMtv^M(\Rd)$,
  \begin{equation*}
    \begin{gathered}
      \sup_{t\in[0,T]} \sup_{\rho\in\cMtv^M(\Rd)} \sup_{x\in \Rd}  \int_{\Rd\setminus\{x\}} |V_t[\rho](x,y)| \eta(x,y) \di \mu(y) \leq C_V,\\ 
      \sup_{x\in \Rd}\int_{\Rd\setminus\{x\}} |V_t[\rho](x,y) - V_t[\sigma](x,y)| \eta(x,y) \di\mu(y) \leq L_V \norm{\rho-\sigma}_{TV}.
    \end{gathered}
  \end{equation*}
Then, there exists a unique measure solution $\rho$ to \eqref{eq:ncl-eq-strong} such that $\rho_0 = \rho^0$.
\end{theorem}
\begin{proof}
  Let $\alpha$ be as in Lemma \ref{lem:S-Lip} and let $a = \alpha T$. If $a< 1$, then the result is direct by applying the well-posedness from Lemma \ref{lem:S-Lip}.
  
  Suppose now $a\geq 1$, write $k$ the integer part of $a$ and let $\tau = 1/(2\alpha)$. Then, by Lemma~\ref{lem:S-Lip}, we know there exists a unique measure solution to \eqref{eq:ncl-eq-strong} on $[0,\tau]$; let us call this solution $\rho^1$ and observe that $\rho^1 \in \mathcal{AC}_{0,\tau}$, where $\mathcal{AC}_{0,\tau}=\AC([0,\tau];\cMtv^M(\Rd))$. Again, applying Lemma~\ref{lem:S-Lip} yields the existence and uniqueness of $\rho^2 \in \mathcal{AC}_{\tau,2\tau}$, the solution to \eqref{eq:ncl-eq-strong} on $[\tau,2\tau]$. By proceeding iteratively, we construct a sequence of solutions
  \begin{equation*}
    \rho^i \in \mathcal{AC}_{(i-1)\tau,i\tau} \quad \text{for all $i \in \{1,\dots,k\}$}, \qquad \rho^{k+1} \in \mathcal{AC}_{k\tau,T}.
  \end{equation*}
  We now define the curve $\rho \in \mathcal{AC}_{0,T} = \ACt$ by
  \begin{equation*}
    \begin{cases} \rho_t = \rho_t^i & \text{for all $t \in [(i-1)\tau,i\tau)$ and $i\in\{1,\dots,k\}$},\\ \rho_t = \rho_t^{k+1} & \text{for all $t \in [k\tau,T]$}, \end{cases}
  \end{equation*}
  which, by construction, is the unique measure solution to \eqref{eq:ncl-eq-strong}.
\end{proof}
We now apply Theorem \ref{thm:existence-strong} to the nonlocal interaction equation studied in \cite{EsPaScSl-arma}, i.e., to the velocity field $v$ as in Example \ref{ex:nlnlie}, but for a more general admissible flux interpolation~$\Phi$. This provides existence and uniqueness of measure solutions to \eqref{eq:nlnl-intro}.
%, which, to our knowledge, is the first result in this direction.
\begin{corollary}[Well-posedness for \eqref{eq:nlnl-intro}]\label{cor:existence-strong-interaction}
  Assume that $\eta$ satisfies
  \begin{equation}\label{eq:cond-f-mu-eta}
    \sup_{x \in \Rd} \int_{\Rd} f(x,y) \eta(x,y) \di\mu(y) < \infty
  \end{equation}
  for some nonnegative measurable function $f \colon \Rd \times \Rd \to \R$.  Let $K \colon \Rd \times \Rd \to \R$ and $P \colon \Rd \to \R$ be such that there exist constants $L_K, L_P>0$ for which
  \begin{equation}\label{eq:K-strong}
    |K(y,z) - K(x,z)| \leq L_K f(x,y), \quad |P(y) - P(x)| \leq L_P f(x,y),
  \end{equation}
for all $x,y,z \in \Rd$. Then, \eqref{eq:nlnl-intro}, whose velocity $V\colon [0,T]\times \cMtv^M(\Rd) \to \cV^{\mathrm{as}}(G)$ we recall is defined for $t\in[0,T]$ and $\sigma \in \cMtv^M(\Rd)$ by
  \begin{equation}
  \label{eq:VKP}
    V_t[\sigma](x,y) = -\dgrad K*\sigma(x,y) - \dgrad P(x,y) \quad \text{for all $(x,y) \in G$},
  \end{equation}
  has a unique measure solution $\rho$ such that $\rho_0 = \rho^0$.
\end{corollary}
\begin{proof}
  We first check that, indeed, $V$ as given in \eqref{eq:VKP} satisfies \eqref{ass:compressibility:V}. 
  \begin{align*}
    |V_t[\rho](x,y)| &= |\dgrad (K*\rho+P)(x,y)| = |K*\rho(y) + P(y) - K*\rho(x) - P(x)|\\
                     &\leq \int_\Rd \left| K(y,z) - K(x,z) \right|  \di|\rho|(z) + |P(y) - P(x)|\\
                     &\leq L_K \int_\Rd f(x,y) \di|\rho|(z) + L_P f(x,y) = (ML_K + L_P) f(x,y);
  \end{align*}
  hence we obtain
  \begin{align*}
    \MoveEqLeft{\sup_{t\in [0,T]}\sup_{\rho\in\cMtv^M(\Rd)}\sup_{x\in\Rd} \int_{\Rd\setminus\{x\}} |V_t[\rho](x,y)| \eta(x,y) \di\mu(y)}\\
    &\phantom{=}\leq (ML_K + L_P) \sup_{x \in \Rd} \int_{\Rd\setminus\{x\}} f(x,y) \eta(x,y) \di\mu(y) <\infty,
  \end{align*}
  which is \eqref{ass:compressibility:V}. Then, we are only left with showing \eqref{ass:Lipschitz:v}. For all $\rho,\sigma \in \cMtv^M(\Rd)$, $t \in [0,T]$ and $(x,y) \in G$, we have
  \begin{align*}
    |V_t[\rho](x,y)-V_t[\sigma](x,y)| &= | \dgrad (K*\rho_t - K*\sigma_t)(x,y) |\\
    &\leq \int_\Rd |K(y,z) - K(x,z)| \di |\rho_t(z) - \sigma_t(z)|\\
    &\leq L_K \norm{\rho_t-\sigma_t}_{TV} f(x,y),
  \end{align*}
  which yields \eqref{ass:Lipschitz:v} and ends the proof.
\end{proof}
Note that choosing the function $f$ in the above corollary to be
  $$
  f(x,y) = |x-y|\vee|x-y|^2 \quad \mbox{for all $x,y \in \Rd$}
  $$
  shows that \cite[Assumption \textbf{(K3)}]{EsPaScSl-arma}, needed for the existence result on weak solutions to \eqref{eq:nlnl-intro} in~\cite[Theorem 3.15]{EsPaScSl-arma}, is stronger than that in \eqref{eq:K-strong} on $K$. On the other hand, the condition \eqref{eq:cond-f-mu-eta}, resulting from this choice of $f$, is a stronger assumption on $\eta$ than \cite[Assumption \textbf{(A1)}]{EsPaScSl-arma}, again needed in Theorem \cite[Theorem 3.15]{EsPaScSl-arma}. Our well-posedness result in Corollary \ref{cor:existence-strong-interaction} thus holds for more general interaction potentials but less general weight functions than our weak existence result in Theorem \cite[Theorem 3.15]{EsPaScSl-arma}. Another interesting example of $f$ which can be chosen in Corollary \ref{cor:existence-strong-interaction} is a constant function, which only imposes $K$ to be a bounded function; in this case, the resulting condition \eqref{eq:cond-f-mu-eta} on $\eta$ is even more restrictive, albeit still reasonable.

\begin{rem}[The case when \texorpdfstring{$\mu$}{the base measure} is atomic]
Let $I \subseteq \N$ be not necessarily finite. Consider $\{x_i\}_{i\in I} \subset \R^d$, $\{m_i\}_{i\in I} \subset [0,\infty)$ and $\mu \in \cM^+(\Rd)$ such that
\begin{equation*}
  \mu = \sum_{i\in I} m_i\delta_{x_i}.
\end{equation*}
Let $V\colon [0,T] \times \cMtv^M(\Rd) \to \cV^{\mathrm{as}}(G)$ satisfy the hypotheses of Theorem \ref{thm:existence-strong}, that is, there exist $C_V,L_V>0$ such that, for all $t\in[0,T]$ and all $\rho,\sigma \in \cMtv^+(\Rd)$, we have
\begin{gather*}
  \sup_{t\in[0,T]}\sup_{\rho\in\cMtv^M(\Rd)}\sup_{x \in \Rd} \sum_{\substack{j\in I\\ x_j\neq x}}^n m_j |V_t[\rho](x,x_j)|\eta(x,x_j) \leq C_V,\\
  \sup_{x \in \Rd} \sum_{j\in I :x_k\neq x}^n m_j |V_t[\rho](x,x_j) - V_t[\sigma](x,x_j)|\eta(x,x_j) \leq L_V \norm{\rho-\sigma}_{TV}.
\end{gather*}

In this case, we know from Theorem \ref{thm:existence-strong} that a unique solution $\rho$ exists on $[0,T]$ such that $\rho_0 = \rho^0$. If $\supp\rho^0\subseteq\supp\mu$, then Proposition \ref{prop:well-posed} entails that the solution stays supported in $\supp\mu$, in particular, $\rho_t\ll\mu$ for a.e. $t\in[0,T]$. If moreover $\Phi$ is jointly antisymmetric, i.e., $\Phi(a,b;-v)=-\Phi(b,a;v)$ for any $a,b,v\in\R$, then \eqref{eq:ncl-strong-sol} rewrites, for any $A\in\cB(\Rd)$ and a.e. $t\in[0,T]$, as
\begin{align*}
  \rho_t[A] = \rho^0[A] -\sum_{i\neq j}\int_0^t \Phi\bra*{r_i(t)m_j, m_i r_j(t), V_s[\rho_s](x_i,x_j)}\eta(x_i,x_j)\di s.
\end{align*}
\end{rem}

\section{\texorpdfstring{$L^p$}{Lp} solutions and positivity preservation}\label{sec:lp}

Let $\rho^0 \in \cMtv^M(\R^d)$ be such that $\rho^0\ll \mu$. In this section, we consider curves in $\AC([0,T];L^1_\mu(\Rd))$ and equip it with the distance
\begin{equation*}
    \|\rho^1-\rho^2\|_{L^\infty([0,T];L^1_\mu(\Rd))}=\sup_{t\in[0,T]}\int_\Rd|\rho^1(x)-\rho^2(x)|\dx\mu(x) \quad \text{for all $\rho^1,\rho^2\in L^1_\mu(\Rd)$}.
\end{equation*}
The advantage of the $L^1_\mu$ setting is that we are able to show positivity preservation of solutions when $\Phi=\Phi_{\mathrm{Upwind}}$, as well as $L^p_\mu$ regularity with $p\in(1,\infty)$.

In this setting, we choose $\lambda=\mu\otimes\mu$ so that the admissible flux from Definition \ref{def:AdmFlux} is given by
  \begin{equation*}
    \dx F^\Phi[\mu;\rho,w](x,y) = \Phi\bra*{\rho(x) ,\rho(y); w(x,y)} \dx (\mu\otimes\mu)(x,y),
  \end{equation*} 
for any $\rho\in L^1_\mu(\Rd)$, $w\in\cV^{\mathrm{as}}(G)$ and $(x,y)\in G$. Assuming that $\Phi$ is jointly antisymmetric, i.e., $\Phi(a,b;-v)=-\Phi(b,a;v)$ for any $a,b,v\in\R$, the nonlocal divergence of $F^\Phi[\mu;\rho,v]$ is given by
\begin{equation*}
    \dgrad\cdot F[\mu;\rho,v](x)=\int_{\Rd\setminus\{x\}}\Phi\bra*{\rho(x) ,\rho(y); v(x,y)}\eta(x,y)\dx\mu(y) \quad \text{for $\mu$-a.e. $x\in\Rd$};
\end{equation*}
properties stated in Proposition \ref{prop:well-posed} still hold. As in Section~\ref{sec:nonl-nonl-inter}, the velocity field may depend on the configuration itself:
\[
v_t(x,y)=V_t[\rho_t](x,y) \quad \text{for all $t\in[0,T]$ and $(x,y)\in G$},
\]
for some $V:[0,T]\times L^1_\mu(\Rd)\to\cV^{\mathrm{as}}(G)$.
The solution map is, for $\mu$-a.e. $x\in\Rd$, given by
\begin{equation}\label{eq:solution-map-l1}
\rho_t(x)=\rho^0(x)-\int_0^t \dgrad\cdot F[\mu;\rho_s,V_s[\rho_s]](x)\dx s.   
\end{equation}

Fix $\rho^0\in L^1_{\mu,M}(\Rd)$. The procedure followed in Section \ref{sec:nonl-nonl-inter} provides a well-posedness result, where, for $M>0$ fixed, we set $L^1_{\mu,M}(\Rd):=\set[\big]{\rho\in L^1_{\mu}(\Rd) : \int_\Rd|\rho(x)|\dx\mu(x)=M}$:
\begin{theorem}[Well-posedness for \eqref{eq:ncl-eq-strong}]\label{thm:existence-strong-l1}
  Let $V\colon [0,T] \times L^1_{\mu,M}(\Rd) \to \cV^{\mathrm{as}}(G)$ and suppose there are constants $C_V,L_V>0$ so that, for all $t\in[0,T]$ and all $\rho,\sigma \in L^1_{\mu,M}(\Rd)$,
  \begin{equation*}
    \begin{gathered}
      \sup_{t\in[0,T]} \sup_{\rho\in L^1_{\mu,M}(\Rd)} \sup_{x\in \Rd}  \int_{\Rd\setminus\{x\}} |V_t[\rho](x,y)| \eta(x,y) \di \mu(y) \leq C_V,\\ 
      \sup_{x\in \Rd}\int_{\Rd\setminus\{x\}} |V_t[\rho](x,y) - V_t[\sigma](x,y)| \eta(x,y) \di\mu(y) \leq L_V \norm{\rho-\sigma}_{L^1_\mu(\Rd)}.
    \end{gathered}
  \end{equation*}
Then, there exists a unique measure solution $\rho$ to \eqref{eq:ncl-eq-strong} satisfying \eqref{eq:solution-map-l1} such that $\rho_0 = \rho^0$.
\end{theorem}
As we now work with densities (with respect to $\mu$), we are able to prove positivity preservation for \eqref{eq:cont-eq-strong} in the case of the upwind flux interpolation; the proof of the result follows the strategy used in \cite{BoyerUpwindFiniteVolume}.
\begin{proposition}[Positivity preservation for \eqref{eq:cont-eq-strong}]
Let $\rho^0$ be nonnegative everywhere and let the assumptions in Theorem \ref{thm:existence-strong-l1} hold. Furthermore, assume that $\Phi\equiv\Phi_{\mathrm{Upwind}}$. Then, the solution $\rho$ to \eqref{eq:cont-eq-strong} is nonnegative a.e., that is, $\rho_t(x)\ge0$ for a.e. $t\in[0,T]$ and $\mu$-a.e. $x\in\Rd$.
\end{proposition}
\begin{proof}
As $\rho$ is absolutely continuous in time, for a.e. $t\in[0,T]$ and $\mu$-a.e. $x\in\Rd$, it holds
\begin{equation*}
\begin{split}
    \partial_t\rho_t(x)&=-\dgrad\cdot F^\Phi[\mu;\rho_t,v_t](x)\\
    &=-\int_{\Rd\setminus\{x\}}v_t(x,y)_+ \eta(x,y)\rho_t(x)\dx\mu(y)\\
    &\quad+\int_{\Rd\setminus\{x\}}v_t(x,y)_-\eta(x,y)\rho_t(y)\dx\mu(y).
\end{split}
\end{equation*}
We denote by $a,A\:[0,T]\to\R$ the maps defined by
\begin{align*}
	a(t):=\sup_{x\in\Rd}\int_{\Rd\setminus\{x\}}|v_t(x,y)_-|\eta(x,y)\dx\mu(y), \quad  
%	a^\gamma(t)&:= \sup_{x\in \supp \gamma\setminus \mu} \int |v_t(x,y)_+| \eta(x,y) \mu(y) \dx\gamma(y) \\
	A(t):=\exp\left(-\int_0^t a(s) \dx s\right),
\end{align*}
and we set $\tilde{\rho}_t(x)=A(t)\rho_t(x)$ for a.e. $t\in[0,T]$ and $\mu$-a.e. $\in\Rd$. In turn, by using $v_+=v+v_-$, we obtain, for $\mu$-a.e. $x\in\Rd$,
\begin{equation*}
    \begin{split}
        \partial_t\tilde{\rho}_t(x)&=A'(t)\rho_t(x)+A(t)\partial_t\rho_t(x)\\
        &=-A(t)a(t)\rho_t(x)-A(t)\dgrad\cdot F^\Phi[\mu;\rho_t,v_t](x)\\
        &=-a(t) \tilde{\rho}_t(x)-A(t)\int_{\Rd\setminus\{x\}}v_t(x,y)_+\eta(x,y)\rho_t(x)\dx\mu(y)\\
        &\quad+A(t)\int_{\Rd\setminus\{x\}}v_t(x,y)_-\eta(x,y)\rho_t(y)\dx\mu(y)\\
        &=- a(t) \tilde{\rho}_t(x)-\int_{\Rd\setminus\{x\}}v_t(x,y)\eta(x,y)\tilde{\rho}_t(x)\dx\mu(y)\\
        &\quad-\int_{\Rd\setminus\{x\}}v_t(x,y)_-\eta(x,y)\tilde{\rho}_t(x)\dx\mu(y)\\
        &\quad+\int_{\Rd\setminus\{x\}}v_t(x,y)_-\eta(x,y)\tilde{\rho}_t(y)\dx\mu(y);
    \end{split}
\end{equation*}
reordering the terms, we get
\begin{equation}\label{eq:auxiliary}
	\begin{split}
		\partial_t \tilde \rho_t(x) &+ \int_{\Rd\setminus\{x\}} v_t(x,y)_- \bra*{\tilde \rho_t(x) - \tilde \rho_t(y)} \eta(x,y) \dx\mu(y) \\
		&+\tilde{\rho}_t(x)\left(a(t)+\int_{\Rd\setminus\{x\}} v_t(x,y)\eta(x,y)\dx\mu(y)\right) = 0,
	\end{split}
\end{equation}
noting that, by definition of $a$, we have
\[
a(t)+\int_{\Rd\setminus\{x\}} v(x,y)\eta(x,y)\dx\mu(y)\ge0.
\]

Let us prove that any supersolution of \eqref{eq:auxiliary} is a.e. nonnegative. Indeed, if this were true, then we would have that the supersolution $\rho_t^\varepsilon:=\tilde{\rho}_t+\varepsilon t=a(t)\rho_t+\varepsilon t\ge0$ $\mu$-a.e., for any $\varepsilon>0$ and a.e. $t\in[0,T]$; and, letting $\varepsilon\to0$, we then would obtain $\rho_t\ge0$ for a.e. $t\in[0,T]$. By contradiction, we thus assume that a supersolution to \eqref{eq:auxiliary}, still denoted by $\tilde{\rho}$, is such that there exists $\tau\in(0,T]$ with
\begin{equation}
    \label{eq:tau-super-sol}
    \inf_{y\in\R^d}\tilde\rho_\tau(y) <0.
\end{equation}
Let $(\tau_k)_k\subset(0,T]$ be defined as $\tau_k = \tau + 1/k$ for all $k>0$ large enough. By the time continuity of $\tilde \rho$ from $[0,T]$ to $L_\mu^1(\Rd)$, we know that, up to a subsequence, $\tilde\rho_{\tau_k} \to \tilde\rho_\tau$ pointwise as $k\to\infty$. Furthermore, let $(x_n^t)_n$ be a minimizing sequence for $\tilde \rho_t$ for all $t\in[0,T]$. Then,
\begin{equation*}
    \tilde \rho_{\tau_k}(x_n^\tau) \xrightarrow[k\to\infty]{} \tilde \rho_\tau(x_n^\tau) \xrightarrow[n\to\infty]{} \inf_{y\in\R^d}\tilde\rho_\tau(y),
\end{equation*}
and similarly, whenever $\tau>0$, for the sequence $(\tau_k')_k\subset(0,T]$ defined by $\tau_k' = \tau - 1/k$ for all $k>0$ large enough. Hence the set $\Delta\subset(0,\infty)$, given by
\begin{equation*}
    \Delta = \left\{ \delta>0 : \forall\, t\in[0,T]\cap(\tau-\delta,\tau+\delta),\; \inf_{y\in\R^d}\tilde\rho_t(y) <0\right\},
\end{equation*}
is nonempty and $\delta_*:=\sup \Delta>0$. Moreover, $\delta_*\le\tau$ since, by assumption, $\tilde\rho_0\geq0$. Setting $\tau_* := \tau - \delta_*\geq0$ and $\tau^* := \min\{T,\tau + \delta_*\}$, we have 
\begin{equation*}
    \inf_{y\in\R^d} \tilde\rho_{\tau_*}(y) \geq 0 \qquad \text{and} \quad \inf_{y\in\R^d} \tilde\rho_t(y) < 0 \quad\text{for all $t\in(\tau_*,\tau^*)$}.
\end{equation*}
For all $h>0$ such that $\tau_*+h < \tau^*$, we have
\begin{equation*}
    \lim_{n\to\infty}\tilde \rho_{\tau_* +h}(x_n^{\tau_*+h}) < 0 \leq \lim_{n\to\infty} \tilde \rho_{\tau_*}(x_n^{\tau_*}) \leq \liminf_{n\to\infty} \tilde \rho_{\tau_*}(x_n^{\tau_*+h}),
\end{equation*}
since $x_n^{\tau_*+h}$ is minimising for $\tilde{\rho}_{\tau_*+h}$ but not necessarily for $\tilde{\rho}_{\tau_*}$, and so 
\begin{equation}
    \label{eq:deriv-super-sol}
    \limsup_{n\to\infty} \left( \tilde \rho_{\tau_* +h}(x_n^{\tau_*+h}) - \tilde \rho_{\tau_*}(x_n^{\tau_*+h}) \right) \leq 0. 
\end{equation}
We find that, for $t_*=\tau_*+h$,
\begin{gather*}
    \limsup_{n\to\infty} \int_{\tau_*}^{\tau_*+h}\int_{\Rd\setminus\{x_n^{t_*}\}} v_{t_*}(x_n^{t_*},y)_- \bra*{\tilde \rho_{t_*}(x_n^{t_*}) - \tilde \rho_{t_*}(y)} \eta(x_n^{t_*},y) \dx\mu(y)\dx t \leq 0,\\
    \limsup_{n\to\infty} \int_{\tau_*}^{\tau_*+h}\!\tilde{\rho}_{t_*}(x_n^{t_*})\left(a(t_*)+\int v_{t_*}(x_n^{t_*},y)\eta(x_n^{t_*},y)\dx\mu(y)\right) \leq 0.
\end{gather*}
Integrating \eqref{eq:auxiliary} between $(\tau_*,\tau_*+h)$ and taking the $\liminf$ as $n\to\infty$, we arrive at
\[
\liminf_{n\to\infty}\left(\tilde{\rho}_{\tau_*+h}(x_n^{\tau_*+h})- \tilde{\rho}_{\tau_*}(x_n^{\tau_*+h})\right)\ge 0,
\]
which contradicts \eqref{eq:deriv-super-sol}. Hence the existence of $\tau$ such that \eqref{eq:tau-super-sol} holds is false and every supersolution to \eqref{eq:auxiliary} must be a.e. nonnegative, which concludes the proof.
\end{proof}

We are also able to prove $L^p$ regularity of solutions for \eqref{eq:cont-eq-strong}:

\begin{proposition}[$L^p$ regularity for \eqref{eq:cont-eq-strong}]
Suppose that $\pderiv{\mu}{x}\in L^\infty(\R^d)$ and $\rho_0$ is nonnegative everywhere with $\rho_0\in L^p(\Rd)$ for some $p\in (1,\infty)$. 
Consider any measurable pair $(\rho,v):[0,T]\to L^1_{\mu,M}(\Rd)\times \cV^{\mathrm{as}}(G)$ satisfying \eqref{eq:solution-map-l1}, with $\Phi\equiv\Phi_{\mathrm{Upwind}}$. Assume that $\eta$ is homogeneous in space, that is, 
%for some $\eta: \R^d\to [0,\infty)$ symmetric it holds
\begin{equation}\label{eq:eta:homogeneous}
\eta(x,y)=\eta(x-y),  \text{ for any } (x,y)\in G.
\end{equation}
Assume there exists a constant $C_v>0$ such that $v:[0,T]\to \cV^{\mathrm{as}}(G)$ satisfies the following uniform translational bound:
\begin{equation}\label{eq:ass:vt:BepsStability}
		\limsup_{\eps\to 0}  \int_0^T \sup_{y\in\Rd}\int_{\Rd} \sup_{h,w\in B_{\eps}(0)} \bra*{ (v_t(x+h,y+w))_- \eta(x,y)}^p \di y \leq C_v.
\end{equation}
Let $\rho$ be the solution to \eqref{eq:nce-expr}. Then, $\rho_t$ is a density with respect to the Lebesgue measure and $\rho_t\in L^1_{\mu,M}(\Rd)\cap L^p(\Rd)$ for all $t\in[0,T]$. Furthermore, for all $t\in [0,T]$, it holds
\begin{equation}\label{eq:est:Lp}
	\sup_{t\in [0,T]} \norm{\rho_t}_{L^p(\Rd)}^p \leq \bra*{ \|\rho_0\|_{L^p(\Rd)}^p+\tilde{C}_v T}\exp\bra*{\frac{T}{q}},
\end{equation}
with $\tilde{C}_v=\frac{C_v}{p}\bra*{p M\norm*{\pderiv{\mu}{x}}_{L^\infty}}^p$.
\end{proposition}
\begin{proof}
Let $\nu$ be a standard mollifier, i.e., a nonnegative and even function in $C^\infty_c(\Rd)$ (the set of smooth, compactly supported functions defined on $\Rd$) such that $\int_\Rd\nu\dx x=1$ and $\supp\nu= B_1(0):=\{x\in\Rd : \norm{x}=1\}$. Fix $\varepsilon>0$ and write $\nu_\varepsilon = \varepsilon^{-d} \nu(\cdot/\varepsilon)$. Also, for any $z\in \R^d$, define the translation operator $\tau^z \colon\R^d\to \R^d$ by $\tau^{z}(h) := h - z$. In particular, set the translated measures $\rho_t^{z}:=\tau^{z}_\#\rho_t$ and $\mu^z:=\tau^z_\#\mu$, where ${}_\#$ stands for the measure-theoretic pushforward. We use the following interplay between translation and convolution: for any $f\in C_b(\Rd)$ (the set of continuous and bounded functions defined on $\Rd$), we have $f*\nu_\varepsilon\in C_b^\infty(\Rd)$, i.e., $f*\nu_\varepsilon\in C_b(\Rd)$ and $f*\nu_\varepsilon$ is smooth, and
\begin{align*}
    \iint_{\Rd\times\Rd} f(h)\nu_\varepsilon(z)\dx\rho_t^z(h)\dx z&=\iint_{\Rd\times\Rd} f(h-z)\nu_\varepsilon(z)\dx\rho_t(h)\dx z\\
    &=\int_\Rd (\nu_\varepsilon*f)(h)\dx\rho_t(h)\\
    &=\iint_{\Rd\times\Rd} \nu_\varepsilon(h-z)f(z)\dx z\dx\rho_t(h)\\
    &=\iint_{\Rd\times\Rd} f(z)\nu_\varepsilon(z-h)\dx\rho_t(h)\dx z\\
    &=\int_{\Rd} f(z)\rho_t^\varepsilon(z)\dx z.
\end{align*}
In particular, for $f\equiv (\rho_t^\eps)^{p-1}$ with $p\geq 1$ and all $t\in[0,T]$, we obtain
\begin{align*}
  \iint_{\Rd\times\Rd} \rho_t^{\eps}(h)^{p-1} \nu_\varepsilon(z)\dx\rho_t^z(h)\dx z&= 
  \int_\Rd \rho_t^\eps(z)^p \dx{z}  = \norm{\rho_t^\eps}_{L^p(\R^d)}^p. 
\end{align*}

Let $\rho^\varepsilon=\rho_t*\nu_\varepsilon$ be the smoothed solution satisfying
\[
\partial_t\rho^\varepsilon_t+\bra[\big]{\dgrad\cdot F^\Phi}*\nu_\varepsilon=0,
\]
where $((\dgrad\cdot F^\Phi)*\nu_\varepsilon)(x)=\int_\Rd\nu_\varepsilon(x-z)\di\dgrad\cdot F^\Phi(z)$ for all $x\in\Rd$. We can compute the time derivative of the $L^p$ norm of $\rho^\varepsilon$: for a.e. $t\in[0,T]$, use \eqref{eq:eta:homogeneous} to get
\begin{align*}
    \MoveEqLeft\frac{\di}{\di{t}}\int_\Rd|\rho^\varepsilon_t|^p\di x = p \int_\Rd\rho^\varepsilon_t(x)^{p-1}\partial_t\rho^\varepsilon_t(x)\di x\\
    &=-p\int_\Rd \rho^\varepsilon_t(x)^{p-1}\bra[\big]{\nu_\varepsilon*\dgrad\cdot F^\Phi}(x)\di x\\
    &=\frac{p}{2}\iint_G\dgrad(\rho_t^\varepsilon)^{p-1}*\nu_\varepsilon\Phi\bra*{\rho_t(x),\rho_t(y); v_t }\eta \dx \mu(x)\dx \mu(y)\\
    &=\frac{p}{2}\int_\Rd \iint_G\dgrad(\rho_t^\varepsilon)^{p-1}(x-z,y-z)\nu_\varepsilon(z) v_t(x,y)_+\eta(x,y)\dx\rho_t(x) \dx\mu(y)\dx z\\
    &\quad - \frac{p}{2}\int_\Rd \iint_G\dgrad(\rho_t^\varepsilon)^{p-1}(x-z,y-z)\nu_\varepsilon(z) v_t(x,y)_-\eta(x,y)\dx\mu(x) \dx \rho_t(y)\dx z\\
    &= - p\int_\Rd \iint_G\dgrad(\rho_t^\varepsilon)^{p-1}(x-z,y-z)\nu_\varepsilon(z) v_t(x,y)_-\eta(x,y)\dx \mu(x) \dx \rho_t(y)\dx z\\
    &= - p\int_\Rd \iint_G\!\dgrad(\rho_t^\varepsilon)^{p-1}(h,w)\nu_\varepsilon(z) v_t(z\!+\!h,z\!+\!w)_-\eta(h,w)\!\dx \mu^z(h)\! \dx \rho_t^z(w)\!\dx z\\
    &\le p\int_\Rd \iint_G\!(\rho_t^\varepsilon)^{p-1}(h)\nu_\varepsilon(z) v_t(z\!+\!h,z\!+\!w)_-\eta(h,w)\!\dx \mu^z(h)\! \dx \rho_t^z(w)\!\dx z=:I.
\end{align*}
To estimate $I$, we use the following variant of Young's inequality: for $p\in (1,\infty)$ and $a,b\in (0,\infty)$, there holds 
\begin{equation}\label{eq:Young}
	a^{p-1} b \leq \frac{a^p}{q} +\frac{b^p}{p}, \quad\text{where } q= \frac{p}{p-1}. 
\end{equation}
Due to \eqref{eq:ass:vt:BepsStability}, for some $\eps_0 >0$ sufficiently small, for all $\eps \in (0,\eps_0)$ and a.e. $t\in[0,T]$, the function $\overline v_t^\eps\colon G \to \R$, defined as
\[
  \overline v_t^\eps(x,y) := \sup_{h,w\in B_{\eps}(0)} \left(v_t(x+h,y+w)\right)_-,
\]
satisfies, for some $C_v^{\eps_0}>0$, the bound
\begin{equation*}
 \sup_{\eps \in (0,\eps_0)}\int_0^T \sup_{x\in\Rd}\int_{\Rd\setminus\{x\}} \bra*{\overline v_t^\eps(x,y) \eta(x,y)}^p \di y \leq C_v^{\eps_0}.
\end{equation*}
Using the bound above, Hölder's inequality and~\eqref{eq:Young}, we get, for a.e. $t\in[0,T]$,
\begin{align*}
I&\leq p \norm*{\pderiv{\mu}{x}}_{L^\infty} \int_\Rd (\rho_t^\varepsilon)^{p-1}(h) \int_{\Rd\setminus \{h\}} \int_\Rd \nu_\varepsilon(z) \overline v_t^\eps(h,w) \eta(h,w)\dx\rho_t^z(w) \dx z \dx h \\
&\leq p \norm*{\pderiv{\mu}{x}}_{L^\infty} \left[ \bra*{\int_\Rd \abs*{\rho_t^\eps(h)}^p \dx{h}}^{\frac{p-1}{p}}\times\right.\\
&\hspace{3.0cm}\left.\times\bra*{\int_\Rd \abs*{\int_{\Rd\setminus\{h\}} \int_\Rd \nu_\varepsilon(z) \overline v_t^\eps(h,w) \eta(h,w)\dx\rho_t^z(w) \dx z}^p \dx h }^\frac{1}{p}\right]\\
&\leq p \norm*{\pderiv{\mu}{x}}_{L^\infty}  \norm{\rho_t^\eps}_{L^p}^{p-1} \bra*{ \int_\Rd \abs*{ \sup_{w\in \R^d} \overline v_t^\eps(h,w) \eta(h,w) \int_{\Rd\setminus\{h\}} \int_\Rd \nu_\eps(z) \dx\rho_t^z(w) \dx z}^p\!\!\! \dx h}^{\frac{1}{p}} \\
% &=p \norm*{\pderiv{\mu}{x}}_{L^\infty}
% \norm{\rho_t^\eps}_{L^p}^{p-1}\norm{\rho_t^\varepsilon}_{L^1} 
% \bra*{ \sup_{w\in \R^d}\int   \abs*{\overline v_t^\eps(h,w) \eta(h,w)}^p \dx h}^{\frac{1}{p}}\\
% &= p \norm*{\pderiv{\mu}{x}}_{L^\infty}
% \rho_0[\Rd]\norm{\rho_t^\eps}_{L^p}^{p-1} 
% \bra*{ \sup_{w\in \R^d} \int   \abs*{\overline v_t^\eps(h,w) \eta(h,w)}^p \dx h}^{\frac{1}{p}}\\
&\le\frac{1}{q}\norm{\rho_t^\eps}_{L^p}^{p}+\frac{1}{p}\bra*{ p \norm*{\pderiv{\mu}{x}}_{L^\infty}
\rho_0[\Rd]}^p\sup_{w\in \R^d} \int_\Rd   \abs*{\overline v_t^\eps(h,w) \eta(h,w)}^p \dx h.
\end{align*}
In turn, we infer
\[
\sup_{t\in [0,T]} \norm{\rho_t^\varepsilon}_{L^p}^p \leq \bra*{ \|\rho_0\|_{L^p(\Rd)}^p+\tilde{C}_vT}\exp\bra*{\frac{T}{q}},
\]
where $\tilde{C}_v=\frac{C_v}{p}\bra*{p \norm*{\pderiv{\mu}{x}}_{L^\infty}\rho_0[\Rd]}^p$. The above inequality ends the proof since, up to a subsequence, we deduce $\rho_t^\varepsilon\rightharpoonup\rho_t$ in $L^p(\Rd)$ for any $t\in[0,T]$, and the stability estimate~\eqref{eq:est:Lp} follows from the arbitrariness of $\eps_0$.
\end{proof}

\subsection*{Acknowledgements}
The authors are deeply grateful to Prof. Dejan Slep\v{c}ev (Carnegie Mellon University) for many enlightening discussions on the contents of the manuscript. AE was supported by the Advanced Grant Nonlocal-CPD (Nonlocal PDEs for Complex Particle Dynamics: Phase Transitions, Patterns and Synchronization) of the European Research Council Executive Agency (ERC) under the European Union’s Horizon 2020 research and innovation programme (grant agreement No. 883363). A considerable part of this work was carried out while AE was a postdoc at FAU Erlangen-N\"{u}rnberg. AE gratefully acknowledge support by the German Science Foundation (DFG) through CRC TR 154  ``Mathematical Modelling, Simulation and Optimization Using the Example of Gas Networks". 
AS is supported by the Deutsche Forschungsgemeinschaft (DFG, German Research Foundation) under Germany's Excellence Strategy EXC 2044 -- 390685587, \emph{Mathematics M\"unster: Dynamics--Geometry--Structure}.

\bibliography{references}
\bibliographystyle{abbrv}

\end{document}